\documentclass{birkjour}

\usepackage{amscd,amssymb,amsmath}
\usepackage{bm}
\usepackage[mathscr]{euscript}
\usepackage{epsfig}
\usepackage{tikz}
\usetikzlibrary{calc}
\usepackage{tikz-cd}
\usepackage{hyperref}
\usepackage[capitalize,nameinlink]{cleveref}
\usepackage{palatino}
\usepackage{enumerate}

\theoremstyle{plain}
\newtheorem{theorem}{Theorem}[section]
\newtheorem{proposition}[theorem]{Proposition}
\newtheorem{lemma}[theorem]{Lemma}
\newtheorem{corollary}[theorem]{Corollary}

\theoremstyle{definition}
\newtheorem{definition}[theorem]{Definition}
\newtheorem{remark}[theorem]{Remark}

\newcommand{\CM}{{\mathbb C}}

\newcommand{\NM}{{\mathbb N}}

\newcommand{\ZM}{{\mathbb Z}}

\newcommand{\Aa}{{\mathcal A}}

\newcommand{\Bb}{{\mathcal B}}
\newcommand{\Dd}{{\mathcal D}}
\newcommand{\Ff}{{\mathcal F}}

\newcommand{\Cc}{{\mathcal C}}

\newcommand{\Ll}{{\mathcal L}}

\newcommand{\Kk}{{\mathcal K}}
\newcommand{\Hh}{{\mathcal H}}

\begin{document}

\title[On a $C^\ast$-Diagonal Generated by the Toric Code]{On a $C^\ast$-Diagonal Generated by the Toric Code}

\author[D. P. Ojito]{Danilo Polo Ojito}
\address{Department of Physics, Universidad de los Andes. Bogota, Colombia \\
	\href{mailto:d.poloo@uniandes.edu.co}{d.poloo@uniandes.edu.co}}

\author[E. Prodan]{Emil Prodan}

\address{Department of Physics and
 Department of Mathematical Sciences, Yeshiva University. New York, NY 10016, USA \\
 \href{mailto:prodan@yu.edu}{prodan@yu.edu}
}

\date{\today}



\begin{abstract}
We study the abelian sub-$C^\ast$-algebra of the $2^\infty$-UHF algebra generated by the star and face operators of Kitaev’s toric code. We show that it is a $C^\ast$-diagonal equivalent to the canonical diagonal of $M_{2^\infty}$.
\end{abstract}

\maketitle

%{\scriptsize \tableofcontents}

\section{Introduction} 
A sub-$C^\ast$-algebra $\Bb$ of a unital $C^\ast$-algebra $\Aa$ is a maximal abelian sub-$C^\ast$-algebra (masa) for $\Aa$ if the relative commutant of $\Bb$ in $\Aa$ is $\Bb$ itself. An element $a \in \Aa$ is said to be a normalizer for the masa $\Bb$ if $a\Bb a^\ast \cup a^\ast \Bb a \subseteq \Bb$, and a masa is called regular if the set of normalizers generates the entire $\Aa$. If in addition $\Bb$ has the unique extension property, meaning that any pure state over $\Bb$ extends uniquely over $\Aa$ as a pure state, then $\Bb$ is a $C^\ast$-diagonal for $\Aa$ \cite{Kum}. The unique extension property automatically implies the existence of a unique faithful conditional expectation from $\Aa$ to $\Bb$ \cite{AndersonTAMS1979}.

$C^\ast$-diagonals and other $C^\ast$-inclusions play important roles in the $C^\ast$-algebra classification program \cite{SchafhauserArxiv2025}. Two $C^\ast$-diagonals for $\Aa$ are said to be (automorphism) equivalent if they are isomorphic as $C^\ast$-algebras and if one of their isomorphisms lifts to an automorphism of $\Aa$. Classification of $C^\ast$-diagonals under this equivalence supplies  valuable information about the structure of individual  $C^*$-algebras. See \cite{Kum,LiTAMS2019}, \cite[Sec.~13]{SchafhauserArxiv2025} and the introductory remarks in \cite{KopsacheilisArxiv}.

In this note, we shall study a  $C^\ast$-diagonal for the quantum spin algebra $M_{2^\infty}$. Despite its structural simplicity, $M_{2^\infty}$ contains many inequivalent $C^\ast$-diagonals and their classification is far from being complete (cf. \cite[Problem~XLVIII]{SchafhauserArxiv2025}). For example, Blackadar demonstrated in \cite{BlackadarAM1990} that $M_{2^\infty}$ includes a family of mutually inequivalent $C^\ast$-diagonals, the spectrum of each of which is the product of a circle and a Cantor set. Recently, a countably infinite family of mutually inequivalent $C^\ast$-diagonals for $M_{2^\infty}$, all with Cantor spectrum, was found by Kopsacheilis and Winter  \cite{KopsacheilisArxiv}. 

In the physics literature, one can find large classes of quantum spin models, dubbed gapped topological phases, with dynamics generated by Hamiltonians containing infinite sequences of commuting terms. These terms are  generated algorithmically by specialized actions of finite-dimensional Hopf algebras on planar spin networks \cite{KitaevAOP2003, BuerschaperJMP2013,LevinPRB2005}. It is  natural to ask whether or not these models generate $C^\ast$-diagonals and how special these diagonals are? This is the question we shall examine in this note, in the simplest nontrivial case, where the  Hopf algebra is  the Drinfeld quantum double of the  group $\ZM_2$,  so that the physical model  is Kitaev's toric code  \cite{Kitaev1997}. 

It was already observed in \cite{AlickiJPA2007} that the  commuting  operators in the toric code generate a masa $\Cc$ for $M_{2^\infty}$ with Cantor spectrum. We shall prove here that this masa is in fact a $C^\ast$-diagonal. The proof engages the concepts of frustration-free nets of quantum spin models and their frustration-free ground states \cite{DET1}. In that context, local topological quantum order (LTQO) is a verifiable property (see definition~\ref{Def:LTQO}) that ensures uniqueness of frustration-free ground states. Our proof reduces to showing that each pure state of $\Cc$ is a frustration-free ground state for a net of frustration-free quantum spin models displaying the LTQO property (see section~\ref{Sec:CDiagonal}).

One might expect the diagonal induced by the toric code to exhibit genuinely exotic features. After all, the toric code is the prototypical example of a topologically ordered quantum spin model: it supports anyonic excitations with fractional statistics, displays long-range entanglement, and exhibits topological degeneracies when deployed on triangulations of surfaces of genus $g\geq 1$ \cite{KitaevAOP2003}. By contrast, the standard diagonal of $M_{2^\infty}$ arises from a purely classical product structure. It is therefore far from obvious that the diagonal generated by the star and face operators of the toric code should be equivalent to the standard one. Yet, the main result of this paper shows that, despite the topological nature of the underlying Hamiltonian, the associated $C^\ast$-diagonal carries no additional structure at the level of automorphism equivalence. While our work proves the automorphism equivalence, finding the explicit equivalence remains an open problem which seems to require fresh ideas. For example, we show that the most obvious isomorphism between the two $C^\ast$-diagonals, already identified and used in \cite{AlickiJPA2007}, does not lift to an automorphism of  $M_{2^\infty}$ (see Proposition~\ref{Prop:AutoM}).

The main elements of our proof can be summarized as follows. First, any $C^\ast$-diagonal inclusion endows the parent $C^\ast$-algebra with a presentation in the form of a twisted groupoid $C^\ast$-algebra. Furthermore, the isomorphism class of the underlying twisted groupoid is a complete invariant for $C^\ast$-diagonal inclusions \cite{Kum}. These statements remain valid in the more general context of Cartan inclusions \cite{Ren1}, which is not needed here but could be relevant for more involved quantum spin models. Therefore, an isomorphism between the twisted groupoids associated with the toric-code diagonal $\Cc$ and the standard diagonal of $M_{2^\infty}$ implies automorphism equivalence of the two $C^\ast$-diagonal inclusions. In fact, since we are dealing only with AF-relations for which all twists are trivial, it is enough to compare the Weyl groupoids of the two $C^\ast$-diagonal inclusions. We do so by computing their Krieger complete invariants \cite{Krieger}, which are identical (see section~\ref{Sec:Equivalence}).

In conclusion, the star and face operators of the toric code generate a $C^\ast$-diagonal for $M_{2^\infty}$, but this $C^\ast$-diagonal is no more special than the standard diagonal, at least from the automorphism equivalence perspective.
\vspace{0.2cm}

\noindent{\bf Acknowledgements.} Our study was supported by the U.S. National Science Foundation through the grant CMMI-2131760, and by U.S. Army Research Office through contract W911NF-23-1-0127. The authors acknowledge fruitful discussions with Nigel Higson and Kang Li.

\section{The toric code model: A short review} 

 The toric code model is defined over the square lattice $\Ll$, whose sets of unoriented edges $e$ and of vertices $v$ are denoted by $E$ and $V$, respectively, and endowed with the discrete topology (hence any compact subset is necessarily finite). The dual graph $\tilde \Ll$ plays an equally important role. Its vertices $\tilde v \in \tilde V$ correspond to the faces of $\Ll$, and its edges $\tilde e \in \tilde E$ intersect those of $\Ll$ transversally. Throughout, we will identify the faces of $\Ll$ with the vertices of $\tilde \Ll$, and use $e\in \tilde v$ to specify that edge $e\in E$ belongs to the boundary of face $\tilde v$. Moreover, we view $\Ll$ and $\tilde \Ll$ as drawn in the same plane, hence statements like $\rho \cap \tilde \rho \neq \emptyset$ make sense for paths of the direct and dual lattices. Given a topological space $X$, we denote by $\Kk(X)$ the family of its compact subsets and, for a subset $S\subseteq X$, we denote its indicator function by $1_S$. We recall that $(\Kk(X),\subseteq)$ are directed sets. We will often consider nets over $\Kk(E)$ and their limits (see {\it e.g.} \eqref{Eq:LTQO} should be interpreted in this context.
 
 Each edge $e \in E$ carries the finite-dimensional Hilbert space $\Hh_e=\CM^2$ and the matrix algebra $\mathcal{A}_e=M_2(\CM)$. For any $\Lambda \in \Kk(E)$, we may define the finite Hilbert space $\Hh_\Lambda:=\bigotimes_{e\in \Lambda}\Hh_e$ and the $C^\ast$-algebra $\Aa_\Lambda=\mathbb{B}(\Hh_\Lambda)$ of linear operators over $\Hh_\Lambda$. The local algebra of observables is given by
$\Aa_{\rm loc}:=\bigcup_{\Lambda\in \Kk(E)} \Aa_{\Lambda}$
with natural inclusions $a\mapsto a\otimes {\bf1}_{\Lambda\setminus \Lambda'}$ for $\Lambda'\subset \Lambda$. The $C^\ast$-closure of this algebra is called the quasilocal-algebra of observables and we shall denote it by $\Aa$. Obviously, it is isomorphic to $M_{2^\infty}$.

We now introduce relevant elements from $\Aa$. In our introductory remarks, we mentioned a set of commuting operators generated by actions of quantum double of $\ZM_2$. The outcome of the procedure is simple enough to be presented without going through the algorithm. Indeed, the commuting operators of the toric code are the star and face operators associated to vertices $v\in V$ and $\tilde v \in \tilde V$,
\begin{equation}\label{Eq:StarFace}
    A_v:=\prod_{e \ni v}\sigma_e^x, \qquad     B_{\tilde v}:=\prod_{e\in \tilde v}\sigma_e^z
\end{equation}
where $\sigma^x_e$ and $\sigma^z_e$ stand for the Pauli matrices from the algebra $\Aa_e$. Notice that $A_v$ and $B_{\tilde v}$ are local commuting symmetries, $A_v^2 = B_{\tilde v}^2 = 1$.  Another set of relevant elements consists of the open ribbon operators
\begin{equation}\label{Eq:Ribbons}
    F_\rho^z := \prod_{e\in \rho} \sigma_e^z, \quad F_{\tilde \rho}^x := \prod_{e \cap \tilde \rho \neq \emptyset} \sigma_e^x,
\end{equation}
where $\rho$ and $\tilde \rho$ are finite continuous and non self-intersecting paths on $\Ll$ and $\tilde \Ll$, respectively. Note that the ribbon operators are also symmetries and, when $\rho$ and $\tilde \rho$ reduce to just one edge, the ribbon operators reduce to the generators $\sigma_e^x$ and $\sigma_e^z$ of $\Aa$. Consequently, the ribbon operators generate the quasi-local algebra $\Aa$. Furthermore, if $\partial \rho$ and $\partial \tilde \rho$ denote the sets of the end vertices of the paths, then the ribbon operators enter into the following relations with the star and face operators:
\begin{equation}\label{Eq:RibbonRelations}
\begin{aligned}
& F_\rho^z A_v = (-1)^{1_{\partial\rho}(v)} A_v F_\rho^z,   \quad  & F_\rho^z B_{\tilde v} =  B_{\tilde v}F_\rho^z   \\
&  F_{\tilde \rho}^x B_{\tilde v} = (-1)^{1_{\partial \tilde \rho}(\tilde v)} B_{\tilde v}F^x_{\tilde \rho}, \quad & F_{\tilde \rho}^x A_v = A_v F^x_{\tilde \rho}
\end{aligned}
\end{equation}

The net of local Hamiltonians for the toric code \cite{KitaevAOP2003} is constructed from the commuting projections
\begin{equation}\label{Eq:Proj}
P_v : =\tfrac{1}{2}({\bf 1}+A_v), \quad P_{\tilde v} :=  \tfrac{1}{2}({\bf 1}+B_{\tilde v})
\end{equation}
as  
\begin{equation}\label{eq: Kitaev hamiltonian}
    H_\Lambda:=\sum_{P_v \in \Aa_\Lambda} P_v^\perp+\sum_{P_{\tilde v} \in \Aa_\Lambda}P_{\tilde v}^\perp,\qquad \Lambda\in \Kk(E).
\end{equation}
The inner-limit derivation corresponding to $\{H_\Lambda\}$ generates a strongly-continuous one-parameter group of automorphisms over $\Aa$. The spectrum of each $H_\Lambda$, as an element of $\Aa_\Lambda$, is positive and includes $0$. The spectral projection corresponding to that lowest eigenvalue can be written for each $\Lambda$ as 
\begin{equation}\label{eq: proj ground state}
    P_\Lambda= \prod_{P_v\in \Aa_\Lambda} P_v  \prod_{P_{\tilde v}\in \Aa_\Lambda} P_{\tilde v}.
\end{equation} 
This defines a net $\{P_\Lambda\}_{\Lambda\in \Kk(E)}$ of local \emph{frustration-free} proper projections on the $AF$-algebra $\Aa$ in the sense of \cite[Definition 2.2]{DET1}. Their key property is that $P_\Lambda \leq P_{\Lambda'}$ for any pair $\Lambda \geq \Lambda' \in \Kk(E)$. We will use both $\{H_\Lambda\}$ and $\{P_\Lambda\}$ interchangeably to encode the Kitaev model.

A state $\omega$ on $\mathcal{A}$ is called a \emph{frustration-free ground state} for a frustration-free net $\{Q_\Lambda\}$ of projections if $\omega(Q_\Lambda) = 1$ for all $\Lambda \in \Kk(E)$ \cite{DET1}. We supply a short argument showing that the toric code has a unique frustration-free ground state. The following property plays a key role here and elsewhere:

\begin{definition}[\cite{DET1}]\label{Def:LTQO}
A net $\{Q_\Lambda\}$ of projections in $\mathcal{A}$ is said to satisfy the \emph{local topological quantum order} (LTQO) condition if, for every local observable $X \in \Aa_{\rm loc}$, one has
\begin{equation}\label{Eq:LTQO}
\lim_{\Lambda} \| Q_\Lambda X Q_\Lambda - \omega_\Lambda(X) Q_\Lambda \| = 0,
\end{equation}
where
\[
\omega_\Lambda(X) = \frac{\mathrm{Tr}(Q_\Lambda X Q_\Lambda)}{\mathrm{Tr}(Q_\Lambda)}
\quad \ {\rm if} \ X \in \Aa_\Lambda
\]
and $0$ otherwise. 
\end{definition}

\begin{remark}
    In the above definition, since $X\in \Aa_{\rm loc}$, there always exists $\Lambda_0 \in \Kk(E)$ such that $X \in \Aa_{\Lambda_0}$, hence the limit \eqref{Eq:LTQO} makes perfect sense.
\end{remark}
\
\begin{theorem}[\cite{DET1}]
If a frustration-free net of proper projections $\{Q_\Lambda\}$ satisfies the LTQO property, then the net converges to a minimal projection in the double dual of $\Aa$. Consequently, there exists a unique frustration-free ground state $\omega$, which is moreover pure, and it is explicitly given by the weak$^\ast$ -limit
\[
\omega = \lim_{\Lambda} \omega_\Lambda.
\]
\end{theorem}

\begin{theorem}\cite[Theorem III.4]{Penneys}\label{teo: unique gs}
    The frustration-free net of projections \eqref{eq: proj ground state} satisfies the LTQO property. More precisely, for  any $X\in \Aa_\Lambda$ there exists $\Delta\supset \Lambda$ such that 
    $$P_\Delta XP_\Delta=\omega_\Delta(X)P_\Delta.$$
\end{theorem}

\begin{remark}
We need the LTQO framework stated in the generality considered in \cite{DET1} because it will be later applied to generalizations of the toric code. Ref.~\cite[Theorem III.4]{Penneys} supplies the crucial check of LTQO condition for the standard toric code.
\end{remark}

Thus, the toric code has a unique frustration-free ground state $\omega$. We can describe the spectral properties of the dynamics generated by $\{H_\Lambda\}$ on the GNS representation of $\Aa$ corresponding to $\omega$. To simplify the notation, we use $A$ instead of $\pi_\omega(A)$ for $A \in \Aa$. Then the mentioned dynamics is generated by the unbounded operator $H = \sum_{v \in V} P_v^\perp+ \sum_{\tilde v \in \tilde V} P_{\tilde v}^\perp$, and:

\begin{proposition}[\cite{KitaevAOP2003}]
    The spectrum of $H$ is $2\NM$ and the lowest eigenvalue is non-degenerate. If $P_\omega$ denotes the spectral projection onto the lowest eigenvalue, then the  projection onto the eigenvalue $2n$ is given by
    $$
    P_{2n} =\sum_{k=0}^n\sum_{\bm v_{2k}, \tilde{\bm v}_{2n-2k}} R(\bm v_{2k}, \tilde{\bm v}_{2n-2k})^\ast \, P_\omega \, R(\bm v_{2k}, \tilde{\bm v}_{2n-2k}),
    $$
    where $\bm v_{2k} = (v_1,\ldots,v_{2k})$ and $\tilde{\bm v}_{2n-2k}=(\tilde v_1,\ldots,\tilde v_{2n-2k})$ are tuples of distinct vertices of the direct and dual lattices, respectively, and
$$
R(\bm v_{2k}, \tilde{\bm v}_{2n-2k}) = \prod_{i=1}^k F_{\rho_i}^z \, \prod_{j=1}^{n-k} F_{\tilde \rho_j}^x,
$$
where the first product is over any set of $k$ ribbons with $\cup_i \partial \rho_i = \bm v_{2k}$, and similarly for the second product.
\end{proposition}

\section{The $C^\ast$-diagonal generated by the toric code}
\label{Sec:CDiagonal}

For convenience, we recall again the definition:
\begin{definition}\label{Def:CDiagonal}
 A sub-$C^\ast$-algebra $\Bb$ of $\Aa$ is called a $C^\ast$-diagonal if:
 \begin{enumerate}
     \item $\Bb$ is a regular masa of $\Aa$, {\rm i.e.} the set of normalizers
$$N_{\Aa}(\Bb):=\big\{ n\in \Aa\mid n^* \Bb n\subset \Bb\;\;\text{and}\;\;n\Bb n^*\subset\Bb\big\} $$
generates $\Aa$ as a $C^*$-algebra.
     \item $\Bb$ has the unique extension property: any pure state over $\Bb$ extends uniquely over $\Aa$ as a pure state.
 \end{enumerate}
\end{definition}

Now, let 
$$
\mathcal{C}=C^*\{A_v,B_{\tilde v}\,|\, v\in V, \ \tilde v \in \tilde V\}
$$ 
be the sub-$C^\ast$-algebra of $\Aa$ generated by the star and face symmetries. In this section, we prove that $\Cc$ is a $C^\ast$-diagonal of $\Aa$. 

\begin{proposition}
    $\mathcal{C}$ is a regular masa for $\Aa$.
\end{proposition}

\begin{proof}
The proof that $\Cc$ is a masa can be found in \cite{AlickiJPA2007}. This also follows from its unique extension property, proven below. Furthermore, from \eqref{Eq:RibbonRelations}, we can see that the conjugations with the ribbon operators leave the algebra $\Cc$ invariant. Since the ribbon operators generate the entire algebra $\Aa$, the regularity follows too.
\end{proof}

For the second property listed in \ref{Def:CDiagonal}, we need to develop several tools. It will be useful to introduce a simplifying notation by setting $W=V\cup \tilde V$, and by placing the star and face operators \eqref{Eq:StarFace} under the map $W\ni w \mapsto S_w$, where $S_w=A_w$ if $w \in V$ and $S_w=B_w$ otherwise.

\begin{proposition}\label{Prop:Omega}
   The Gelfand spectrum ${\rm Spec}(\Cc)$ of $\Cc$ is homeomorphic to the Cantor set
   \begin{equation}
    \Omega=\{-1,1\}^W=\big\{f\colon W \to \{-1,1\}\big\}, \quad W=V \cup \tilde V.
\end{equation}  
\end{proposition}

\begin{proof}
Notice that any multiplicative linear functional on $\Cc$ is completely determined by its values on the generators $S_w$. Since $\{S_w\}_{w\in W}$ is a family of commuting self-adjoint symmetries, every such functional $\pi$ must satisfy $\pi(S_w)\in\{-1,1\}$ for all $w\in W$. Therefore, there is a natural identification ${\rm Spec}(\Cc)\subseteq\{-1,1\}^W$.
 On the other hand, any function $W\to \{-1,1\}$ can be extended to a multiplicative linear functional over $C^\ast(S_w, \, w \in W)$ since there are no relations between the $S_w.$
\end{proof}

The above homeomorphism is explicitly provided by the map 
\[
    {\rm Spec}(\Cc) \ni \pi \mapsto (\pi(S_w))_{w \in W}\in \Omega.
\]
Under this identification, $S_w$ becomes the function $S_w(f)=f(w)$, for $f\in \Omega$ and $w\in W$. We shall denote by $\omega_f$ the pure state over $\mathcal{C}$ given by evaluation at $f\in \Omega.$ The unique frustration-free ground state $\omega$ of the toric code model satisfies the condition $\omega(S_w)=1$ for all $w\in W$. As a result, one can identify the restriction of $\omega$ to $\mathcal{C}$ with the evaluation ${\rm ev}_{1_\Omega}$ at the configuration $ 1_\Omega\in \Omega$ that takes the value $1$ for all $w\in W$. In fact, it is also true that $\omega$ is the unique pure extension of  ${\rm ev}_{1_\Omega}$ \cite[Proposition 2.2]{Naijk11}. Our task is to show that all pure states of $\mathcal{C}$ share the same property. To this end, we associate to any $f\in \Omega$ the frustration-free net of proper projections
\begin{equation}\label{eq: projections}
    P_\Lambda(f)=\prod_{Q_w(f) \in \Aa_\Lambda} Q_w(f) \in \Aa_\Lambda, \quad Q_w(f)= \tfrac{1}{2}(1+f(w)S_w).
\end{equation} 
Note that $Q_w(f)$ are commuting projections that reduce to $P_w$ from \eqref{Eq:Proj} if $f=1_\Omega$.  Furthermore, if $f_w\in\Omega$ is obtained from $1_\Omega$ by flipping only the value at $w$, then $
Q_w(f_w)=1-P_w.$

\begin{proposition} 
 Let $f \in \Omega$. Then any extension over $\Aa$ of the pure state $\omega_f$ is a frustration-free ground state for the net  $\{P_\Lambda(f)\}$.
\end{proposition}

\begin{proof}
  This is a consequence of the fact that $P_\Lambda(f)$'s reside inside $\Cc$, hence $\eta\big(P_\Lambda(f)\big)=\omega_f\big(P_\Lambda(f)\big)$ for any extension $\eta$ of $\omega_f$. The statement then follows because $\omega_f$, being pure, is a character of $\Cc$ and $\omega_f\big(Q_w(f)\big)=1$ for all $w \in W$.
\end{proof}

If the net \eqref{eq: projections} displays the LTQO property, then it has a unique frustration-free ground state, hence $\omega_f$ must have a unique extension over $\Aa$.  We prepare to show that this is indeed the case. A symmetry on $\Aa$ will be an automorphism $\alpha$ such that $\alpha^2={\rm id}$. Following \cite{BlackadarAM1990}, we say that automorphism $\alpha$  on $\Aa$ is locally representable if $\alpha(\Aa_\Lambda)= \Aa_\Lambda$ for all $\Lambda\in \Kk(E)$. The following technical Lemma states that $\Aa$ admits a family of locally representable symmetries satisfying a number of useful conditions.

\begin{lemma}\label{lem: auto}
There is a family of locally representable symmetries $\{\alpha_w\}_{w \in W}$ which commute pairwise under the composition and $\alpha_w(P_{w})=1-P_w$ and $\alpha_w(P_{w'})=P_{w'}$ for any pair $w \neq w'$ from $W$.
\end{lemma}

\begin{proof}
 Let $\zeta_v$ be a semi-infinite straight horizontal path of edges in $\Ll$, starting at vertex $v$ and progressing to the right, and let $\zeta_v^n$ denote the finite sub-path consisting of its first $n$ edges. For each element $X \in \Aa$ define
$$\alpha_v(X)=\lim_n F^z_{\zeta_v^n} X F^z_{\zeta_v^n}.$$
It is clear that the above limit converges in the operator norm and $\alpha_v^2={\rm id}$, and consequently $\alpha_v$ is a locally representable automorphism $\Aa.$ Similarly, let $\tilde \zeta_{\tilde v}$ be a path on the dual graph $\tilde \Ll$, starting at the vertex $\tilde v$ and progressing to the right, and let $\tilde \zeta_{\tilde v}^n$ be the finite sub-path made up of its first $n$ edges. Then 
$$\alpha_{\tilde v}(X)=\lim_n F_{\tilde \zeta_{\tilde v}^n}^x X F^x_{\tilde \zeta_{\tilde v}^n},$$
defines again a locally representable symmetry on $\Aa$. Then $\alpha_w$'s commute pairwise because all paths $\zeta_v$ and $\tilde \zeta_{\tilde v}$ are non-intersecting and the stated actions of $\alpha_w$'s on $P_w$'s follow from \eqref{Eq:RibbonRelations}.
\end{proof}

\begin{theorem}
    The net of projections $\{P_\Lambda(f)\}$ satisfies the LTQO property for any $f\in \Omega$.
\end{theorem}
\begin{proof}
    Denote by $P_\Lambda(1_\Omega)\equiv P_\Lambda$ and recall that $\omega_{1_\Omega}\equiv \omega$, the unique frustration-free ground state of the standard toric code. Let $X\in \Aa_{\rm loc}$ be a local observable and $\Lambda$ such that $X\in \Aa_\Lambda.$ From Theorem \ref{teo: unique gs},  there exist $\Delta\supset \Lambda$ such that $P_\Delta YP_\Delta=\omega_\Delta(Y)P_\Delta$ for all $Y\in \Aa_\Lambda$. 
We consider the locally representable symmetry
$$\alpha=\prod_{w\in W^f_\Delta}\alpha_w,\qquad W_\Delta^f=\big\{w\in W\mid f(w)=-1\;\text{and}\; Q_w
(f)\in \Aa_\Delta\big\},$$
where $\alpha_w$'s are the symmetries given in Lemma \ref{lem: auto}.  Observe that $\alpha$ is well-defined since  only finitely many projections $Q_w(f)$ belong to $\Aa_\Delta$, hence $W_\Delta^f$ is finite. Moreover, it follows that $\alpha(P_\Delta(f))=P_\Delta$ and consequently,
\begin{align*}
   \alpha\big(P_\Delta(f) XP_\Delta(f)\big)= P_\Delta \alpha(X)P_\Delta=\omega_\Delta(\alpha(X)) P_\Delta.
\end{align*}
Since $\alpha$ is a symmetry, one gets
\[P_\Delta(f)XP_\Delta(f)=\omega_\Delta(\alpha(X))P_\Delta(f).
\]
The above completes the proof because $X$ is an arbitrary local observable and 
\[\omega_\Delta(\alpha(X))=\frac{1}{{\rm Tr}(P_\Delta)} {\rm Tr}\big(P_\Delta\alpha(X)P_\Delta\big)=\frac{1}{{\rm Tr}(P_\Delta(f))} {\rm Tr}\big(P_\Delta(f)XP_\Delta(f)\big)
\]
where in the last step we used the fact that the local trace is $\alpha$ invariant.
\end{proof}
\begin{corollary}\label{coro: unique extension}
    $\mathcal{C}$ has the unique extension property.
\end{corollary}

\begin{remark}\label{Rem:CondExp}
  Theorem~3.4 in \cite{AndersonTAMS1979} and the above result assure us of the existence of a unique conditional expectation $E\colon \Aa\to \Cc$. This conditional expectation can be described explicitly: for each $a \in \mathcal{A}$, the compression $E(a)$ is determined by the  continuous function 
$E(a)(f) = \tilde{\omega}_f(a),$ where $f\in \Omega$ and  $\tilde{\omega}_f$ denotes the unique pure extension of the state $\omega_f$ on $\mathcal{C}$ to $\mathcal{A}$.
\end{remark} 

\section{Equivalence class of the $C^\ast$-diagonal}
\label{Sec:Equivalence}

In our setting, the standard $C^\ast$-diagonal $\Dd$  of $M_{2^\infty} \simeq \Aa$ takes the form
\begin{equation}
    \Dd=C^*\{\sigma_e^z, \ e\in E\}.
\end{equation} In this section we prove that the $C^\ast$-diagonals $\mathcal{C}$ and $\mathcal{D}$ are equivalent. 

We achieve this by comparing the groupoids associated to various $C^\ast$-inclusions. We recall that a twist of a locally compact, Hausdorff, étale, topologically principal groupoid with unit space $G^{0}$, is a central extension of groupoids
\[
\mathbb{T}\times G^{0}
\;\longrightarrow\;
\Sigma
\;\xrightarrow{\;\;}
G,
\]
where $\mathbb{T}$ acts freely on $\Sigma$ so that $\Sigma/\mathbb{T}\cong G$.
The twist $\Sigma$ is said to be \emph{trivial} if it is isomorphic to $G\times\mathbb{T}$.  Given a $C^\ast$-diagonal inclusion, or more generally, a Cartan inclusion $(A,B)$, the fundamental results of \cite{Kum,Ren1} state that there exists a groupoid $G$, called the Weyl groupoid of the pair, together with a twist $\Sigma$ of $G$ such that
\[
(A,B)\;\cong\;\bigl(C_r^*(G,\Sigma),\, C_0(G^{0})\bigr).
\]
Moreover, the twist $\Sigma$ is uniquely determined by $(A,B)$ up to conjugation. In other words, the twist is a complete invariant of the $C^\ast$-inclusion.

In the construction of the Weyl groupoid of a $C^\ast$-diagonal inclusion, it is shown that each normalizer induces a partial homeomorphism on the Gelfand spectrum of the $C^\ast$-diagonal \cite{Kum}. More concretely, if $X={\rm Spec}(B)$ and $n\in N_A(B)$ is a normalizer, then
$$
\operatorname{dom}(n):=\big\{\phi\in X:\phi(n^*n)\neq 0\big\},
\qquad \operatorname{ran}(n):=\big\{\phi\in X:\phi(nn^*)\neq 0\big\}
$$
are open subsets of $X$, and there is a unique homeomorphism
$
\alpha_n\colon \operatorname{dom}(n)\to \operatorname{ran}(n)
$
characterized by
$$
\phi(n^\ast bn)=\alpha_n(\phi)(b)\phi(n^\ast n),
\qquad b\in B,$$
for every $\phi\in\operatorname{dom}(n)$. Thus, an element of the Weyl groupoid is represented by a pair $(n,\phi)$, where $n$ is a normalizer and $\phi\in\operatorname{dom}(n)$,  modulo the equivalence relation identifying pairs whose associated partial homeomorphisms agree on some neighborhood of $\phi$. If a normalizer happens to be unitary, then it induces a full homeomorphism of $X$. We refer to \cite{Kum, Ren1} for the complete construction of Weyl groupoids. 

Now, from \eqref{Eq:RibbonRelations}, we can see that the ribbon operators are unitary normalizers of $\Cc$. Our first goal is to explicitly characterize the group of their induced homeomorphisms on ${\rm Spec}(\Cc)$. Consider two copies of the edge set $E$, denoted by $E_x$ and $E_z$, and define the locally compact abelian group $\Gamma$ consisting of the set of functions $\gamma : E_x \cup E_z \to \ZM_2$ with compact support, {\it i.e.}, $\gamma^{-1}(-1)\in \Kk(E_x \cup E_z)$, equipped with the final topology. Here and throughout this section, $\ZM_2$ is written multiplicatively. We recall that $\Omega$ was defined as $\{-1,1\}^W$ and we endow the set $\{-1,1\}$ with the group structure of $\ZM_2$. Then, using the embeddings $\mathfrak j_i :E \to E_x \cup E_z$ with $\mathfrak j_i(E)=E_i$ for $i=x,z$, we define a map $\partial\colon \Gamma\to \Omega$ by
\[
\partial\gamma(v)
= \prod_{e \ni v} \gamma\big(\mathfrak j_z(e)\big),
\quad
\partial\gamma(\tilde v)
= \prod_{e \in \tilde v} \gamma\big(\mathfrak j_x(e)\big),
\]
for any  $v \in V$ and $\tilde{v}\in \tilde V$. This yields an element $\partial\gamma\in \Omega$ with compact support. Since $\partial(\gamma_1 
\cdot \gamma_2)=\partial \gamma_1 \cdot \partial \gamma_2$, the image $\partial \Gamma$ is an abelian group inside $\Omega$, and we define the action $\partial \Gamma\times \Omega\to \Omega$ by
\[
(\partial \gamma,f)\;\longmapsto \;\alpha_{\partial \gamma}(f):=\partial\gamma \cdot f.
\]
We equip $\partial \Gamma$ with the final topology.

\begin{proposition}
    The action of $\partial \Gamma$ on $\Omega$ is free and minimal.
\end{proposition}

\begin{proof}
Observe that the image of $\Gamma$ under $\partial$ coincides with the set of functions $f : W \to \ZM_2$ for which $f^{-1}(-1) \cap V$ and $f^{-1}(-1) \cap \tilde V$ are finite and have even cardinalities. This set of functions is dense in $\Omega$, hence the orbit $(\partial \Gamma) \cdot 1_\Omega = \partial \Gamma$ is dense in $\Omega$. If $f \in \Omega$ is any other function, then point-wise multiplication by $f$ defines a homeomorphism over $\Omega$, since $f \cdot f = 1_\Omega$ and point-wise multiplication by $f$ is continuous over $\Omega$. We have
\[
(\partial \Gamma) \cdot f = (\partial \Gamma)\cdot (f \cdot 1_\Omega) = f \cdot (\partial \Gamma)1_\Omega = f \cdot \partial \Gamma,
\]
which is again dense in $\Omega$. As such, the action of $\partial \Gamma$ on $\Omega$ is minimal. The action is also free because $\partial \gamma \cdot f = f$ implies $\partial \gamma \cdot f \cdot f = f \cdot f$, or $\partial \gamma =1_\Omega$.
\end{proof}

\begin{proposition}\label{Prop:AlphaAction}
Let $\Ff$ be the subgroup of unitary elements of $\Aa$ generated by the ribbon operators \eqref{Eq:Ribbons}. Then there exists a group epimorphism $\Ff \ni u \mapsto \partial \gamma_u \in \partial \Gamma$, such that  
\begin{equation}\label{Eq:GammaU}
u^\ast \eta(Q) u = \eta(Q \circ \alpha_{\partial \gamma_u}), \quad \forall \ Q\in C(\Omega).
\end{equation}
Here, $\eta :C(\Omega) \to \Cc$ is the isomorphism established in \ref{Prop:Omega}
\end{proposition}

\begin{proof} The assignment $u \mapsto {\rm Ad}_u$ is a group morphism from $\Ff$ to the group ${\rm Aut}(\Cc)$ of automorphisms of $\Cc$. We denote by ${\rm Ad}_\Ff$ the image of this morphism. Note that, if the morphism $\Ff \to \partial \Gamma$ exists, then relation \eqref{Eq:GammaU} implies that it factors through the automorphism $\Ff \to {\rm Ad}_\Ff$. We will show that this is indeed the case by proving ${\rm Ad}_\Ff \simeq \partial \Gamma$. Note that, although the ribbon operators display non-trivial commutation relations,
\[
F_\rho^z F_{\tilde \rho}^x = (-1)^{|\rho \cap \tilde \rho|} F_{\tilde \rho}^x F_\rho^z,
\]
the non-commutative character disappears under conjugation:
\[
{\rm Ad}_{F_\rho^z} \circ {\rm Ad}_{F_{\tilde \rho}^x} = (-1)^{2|\rho \cap \tilde \rho|} {\rm Ad}_{F_{\tilde \rho}^x}\circ {\rm Ad}_{F_\rho^z}= {\rm Ad}_{F_{\tilde \rho}^x} \circ {\rm Ad}_{F_\rho^z}.
\]
Furthermore, since the ribbon operators are products of $\sigma_e^x$ and $\sigma_{e'}^z$, which are ribbon operators as well, ${\rm Ad}_\Ff$ is generated by ${\rm Ad}_{\sigma_e^z}$ and ${\rm Ad}_{\sigma_{e'}^x}$, $e,e' \in E$. Now, based on \eqref{Eq:RibbonRelations}, for any $w \in W=V\cup \tilde V$,
\begin{equation}\label{Eq:Action1}
    {\rm Ad}_{\sigma_e^z}(S_w)=(-1)^{1_{\partial e}(w) 1_V(w)} S_w, \quad  {\rm Ad}_{\sigma_e^x}(S_w)=(-1)^{1_{\partial e}(w) 1_{\tilde V}(w)} S_w.
\end{equation}
On the other hand, the functions $\gamma_e^x,\gamma_e^z \in \Gamma$, defined by $(\gamma_e^i)^{-1}(-1)=\{\mathfrak j_i(e)\}$, square to the identity, generate $\Gamma$, and
\[
\partial \gamma_e^z(w)=(-1)^{1_{\partial e}(w) 1_V(w)}, \quad \partial \gamma_e^x(w)=(-1)^{1_{\partial e}(w) 1_{\tilde V}(w)}.
\]
They induce the following actions on $C(\Omega)$:
\begin{equation}\label{Eq:Action2}
\begin{aligned}
    & (\alpha_{\partial \gamma_e^z}\, f )(w)=(-1)^{1_{\partial e}(w) 1_V(w)}f(w), \\
    & (\alpha_{\partial \gamma_e^x}\, f )(w)=(-1)^{1_{\partial e}(w) 1_{\tilde V}(w)} f(w).
\end{aligned}   
\end{equation}
As we already mentioned, when viewed as elements of $C(\Omega)$, the symmetries $S_w$ become the functions $S_w(f) = f(w)$ for all $f\in \Omega$ and, as such, 
\begin{equation}\label{Eq:Action3}
    {\rm Ad}_{\sigma_e^i}(S_w)=\eta\big( \eta^{-1}(S_w) \circ \alpha_{\partial \gamma_e^i}\big),\quad i=x,z, \quad e \in E.
\end{equation}
Since $\Cc$ is generated by $S_w$'s, we can conclude
\begin{equation}\label{Eq:Action4}
    {\rm Ad}_{\sigma_e^i}(\eta(Q))=\eta\big( Q \circ \alpha_{\partial \gamma_e^i}\big),\quad i=x,z, \quad Q\in C(\Omega).
\end{equation}

We now show that the assignment ${\rm Ad}_{\sigma_e^i} \mapsto \partial \gamma_e^i$ supplies the group isomorphism ${\rm Ad}_\Ff \simeq \partial \Gamma$. First, we need to address the existing relations among the generators of ${\rm Ad}_\Ff$. Let $\prod_{k=1}^n {\rm Ad}_{\sigma_{e_k}^{i_k}}=1$ be such relation. Then, from iterations of \eqref{Eq:Action4}, 
\[
Q \circ \alpha_{\partial \gamma_{e_1}^{i_1}} \circ \cdots \circ \alpha_{\partial \gamma_{e_n}^{i_n}} =Q \circ \alpha_{\partial \gamma_{e_1}^{i_1} \cdots \partial \gamma_{e_n}^{i_n}} = Q,
\]
for all $Q\in C(\Omega)$. This can be so if and only if $\partial \gamma_{e_1}^{i_1} \cdots \partial \gamma_{e_n}^{i_n} = 1$. Thus, the proposed assignment respects the relations among the generators of ${\rm Ad}_\Ff$, making it into a group morphism. Relation \eqref{Eq:Action4} also assures us that this morphism is injective.\end{proof}

 Let $G_\mathcal{C}=\Omega\rtimes \partial \Gamma$ be the transformation groupoid corresponding to the action $\alpha$, and let $\mathcal{C}\rtimes \partial \Gamma$ denote the associated crossed product algebra. Proposition~\ref{Prop:AlphaAction} assures us that $G_\mathcal{C}$ is a subgroupoid of the Weyl groupoid $G(\mathcal{C})$ associated with the pair $(\mathcal{A},\mathcal{C})$. 
 
 %An important observation is that $G_\Cc$ belongs to the class of AF-groupoids \cite{GTS, Krieger,RenaultBook}.

\begin{proposition}\label{prop: AF}
The groupoid $G_\mathcal{C}$ is an $AF$-relation in the sense of \cite[Definition~3.7]{GTS}. In particular,
$G_\mathcal{C}$ is principal, étale, Hausdorff, locally compact, second countable.
\end{proposition}

\begin{proof}
$\partial\Gamma$ is locally finite since every finitely generated subgroup of $\partial \Gamma$ is finite. The claim then follows from \cite[Theorem~3.8]{GTS} since $\partial \Gamma$-action on $\Omega$ is free and minimal. 
\end{proof}

By a result of Krieger \cite{Krieger} (see also \cite[Theorem~4.10]{Mat}), for an $AF$-relation $G$ with Cantor unit space $G^0$, the ordered dimension group
\[
\big(H_0(G),H_0(G)^+,[1_{C(G^0)}]\big),
\]
as defined in \cite[Definition~3.1]{Mat}, supplies a complete invariant. A computation of it for $G_\Cc$ reveals:

\begin{lemma}\label{Lemma:CDEquiv}
Let $G(\mathcal{D})$ be the Weyl groupoid of the pair $(\mathcal{A},\mathcal{D})$.
Then $G_\mathcal{C}\simeq G(\mathcal{D})$ as groupoids.
Consequently, $\mathcal{C}\rtimes \partial \Gamma \simeq \mathcal{A}$.
\end{lemma}

\begin{proof}
Since both groupoids are $AF$-relations, it suffices to compute the ordered dimension group
for $G_\mathcal{C}$ and verify that it agrees with that of $G(\Dd)$. It is known that 
\[
(H_0(G(\mathcal{D})),H_0(G(\mathcal{D}))^+,[1_{G(\mathcal{D})^0}]) \simeq (\ZM[\tfrac{1}{2}],\ZM_+[\tfrac{1}{2}], 1).
\]
We recall from \cite[Pg.~5]{Mat} that $H_0(G_\Cc)$ is equal to the quotient of $C(\Omega,\ZM)$ by its subgroup generated by $Q-Q \circ \alpha_{\partial \gamma}$, for $Q \in C(\Omega,\ZM)$ and $\partial \gamma \in \partial \Gamma$. We will need a few facts about the family of cylinder subsets
\[
\Omega(K,\epsilon):= \{f \in \Omega\, |\, f|_K = \epsilon\}, \quad K\in \Kk(W), \quad \epsilon : K \to \ZM_2,
\]
which form a basis for the topology of $\Omega$. Finite intersections of cylinder sets are again cylinder sets. If $K\subset K'$, then any $\Omega(K,\epsilon)$ can be refined as
\[
\Omega(K,\epsilon) = \bigcup_{\delta:K'\setminus K\to \ZM_2} \Omega(K', \epsilon \vee \delta).
\]
Lastly, every clopen subset of $\Omega$ is a finite union of cylinder sets.

Now, any $Q \in C(\Omega,\ZM)$ can be sliced by its level sets,
\begin{equation}\label{Eq:QLevelSets}
Q=\sum_{k \in \ZM} k \; 1_{\Omega_k}, \quad \Omega_k =Q^{-1}(k).
\end{equation}
Each $\Omega_k$ is a clopen subset and there are only a finite number of $\Omega_k$'s that are not void, because $Q$ is bounded. Then $\Omega_k=\bigcup_{j=1}^{n_k} \Omega(K_j,\epsilon_j)$ and the inclusion-exclusion formula
\[
1_{\Omega_k} = \sum_{\emptyset \neq J \subset \{1,\ldots,n_k\}} (- 1)^{|J|+1} \, 1_{\bigcap_{j\in J}\Omega(K_j,\epsilon_j)}
\]
together with the facts mentioned above assure us that 
\[
Q=\sum_{n=1}^N a_n \, 1_{\Omega(K_n,\epsilon_n)}, \quad N\in \NM, \quad a_n \in \ZM. 
\]
By taking $K_Q=\cup K_n$ and using refinements of the cylinder sets, we find
\[
Q=\sum_{n=1}^{N_Q} a'_n \, 1_{\Omega(K_Q,\epsilon_n)}, \quad N_Q \in \NM, \quad a'_n \in \ZM. 
\]
Now, for a pair of cylinder sets $\Omega(K,\epsilon)$ and $\Omega(K,\epsilon')$, consider any $\partial \gamma \in \partial \Gamma$ such that $\partial \gamma|_K = \epsilon\cdot \epsilon'$. Such element always exists. Then 
\[
\alpha_{\partial \gamma} \big(\Omega(K,\epsilon)\big) \subseteq \Omega(K,\epsilon'), \quad \alpha_{\partial \gamma}  \big ( \Omega(K,\epsilon') \big ) \subseteq \Omega(K,\epsilon),
\]
and, since $\alpha_{\partial \gamma}$ is its inverse, $\alpha_{\partial \gamma}$ establishes a homeomorphism between $\Omega(K,\epsilon)$ and $\Omega(K,\epsilon')$. As such, $1_{\Omega(K,\epsilon)}$ and $1_{\Omega(K,\epsilon')}$ belong to the same class in $H_0(G_\Cc)$. Therefore, the class of any $Q\in C(\Omega,\ZM)$ can be represented as
\begin{equation}\label{Eq:QExpansion}
    [Q] = a_{\Omega(K_Q,\epsilon_Q)}\, [1_{\Omega(K_Q,\epsilon_Q)}],
\end{equation}
for some cylinder set and integer coefficient.

Now, let $\mu$ be the Bernoulli probability measure on $\Omega$, \textit{i.e.}, $\mu=\bigotimes_W \delta$ with $\delta$ the Haar measure on $\ZM_2$. It is $\partial\Gamma$-invariant, and therefore the map
\[
\varphi\colon H_0(G_\mathcal{C})\to \mathbb{R},
\qquad
\varphi([Q])=\int_\Omega Q \,\mathrm{d}\mu,
\]
is a group homomorphism \cite[Section~6]{Mat}. Since $\varphi(1_{\Omega(K,\epsilon)})= 2^{-|K|}$, it follows from \eqref{Eq:QExpansion} that $\varphi$ is injective and that its image coincides with $\ZM[\frac{1}{2}]$. Furthermore, $H_0(G_\Cc)^+$ coincides with the classes of those $Q\in C(\Omega,\ZM)$ for which the sum \eqref{Eq:QLevelSets} restricts to $\ZM_+$ and, as such, its image through $\varphi$ is $\ZM_+[\frac{1}{2}]$. Lastly, $1_\Omega=1_{\Omega(
\emptyset,\epsilon_\emptyset)}$, hence this cycle is sent by $\varphi$ into $1\in \ZM[\frac{1}{2}]$. \end{proof}

 To complete the main statement of the section, we will use \cite[Proposition 4.13]{Ren1} and, for clarity, we reproduce the relevant part of it here:

\begin{proposition}\cite[Proposition 4.13]{Ren1}\label{Prop:Renault413}
 Let $(G,\Sigma)$ be a twisted etal\'e Hausdorff locally compact second countable groupoid. Let $A=C^\ast_r(G,\Sigma)$ be its reduced $C^\ast$-algebra and $B=C_0(G^0)$. Assume that $G$ is topologically principal. Then the Weyl groupoid of the $C^\ast$-inclusion $(A,B)$ is canonically isomorphic to $G$.
\end{proposition}
We also need the following important observation about AF-relations.
\begin{proposition}\label{prop: twist}
    If $G$ is an AF-relation, then any twist $\Sigma$ of $G$ is trivial.
\end{proposition}
\begin{proof}
    Let us first notice that $G$ is a transformation groupoid. Indeed, any AF-relation is an étale equivalence relation on the Cantor space $G^0$ in the sense of \cite[Definition~2.1]{GTS}.
Therefore, by \cite[Proposition~2.3]{GTS}, there exists a countable group
$\mathcal G$ acting on $G^0$ such that
$G \simeq G^0 \rtimes \mathcal G$ as groupoids.
Since $G$ is an AF-relation, the action of $\mathcal G$ on $G^0$ is necessarily
free and minimal. It then follows from \cite[Theorem~3.8]{GTS} that $\mathcal G$ is locally finite. Consequently, $G$ falls within the framework of \cite[Proposition~7.1]{LiArxiv},
and hence the  twist is always trivial by
\cite[Proposition~6.3]{LiArxiv}.
\end{proof}

With the above ingredients in place, we are now in a position to state the main result of this note.
\begin{theorem}
   The $C^\ast$-diagonals $\mathcal{C}$ and $\mathcal{D}$ are equivalent. 
\end{theorem}

\begin{proof}
    Take $(G,\Sigma)$ in \ref{Prop:Renault413} to be $G_\Cc$ with the trivial twist, as every twist over $G_\Cc$ is trivial by Propositions \ref{prop: AF} and \ref{prop: twist}. Then Lemma \ref{Lemma:CDEquiv} and Proposition \ref{Prop:Renault413} assure us that $G_\Cc$ is isomorphic to the Weyl groupoid $G(\mathcal{C})$ of the $C^\ast$-diagonal $\Cc \subset \Aa$ and, in particular, that $G(\mathcal{C})$ has only trivial twists. Thus, the result follows by applying again Lemma \ref{Lemma:CDEquiv} and the fact that the Weyl twists, which reduce to the Weyl groupoids in our case, are complete invariants for $C^\ast$-diagonals \cite{Kum,Ren1}.
\end{proof}

   Finding an automorphism that implements the predicted equivalence appears to be non-trivial. Indeed, we can show how the most obvious isomorphism $\alpha\colon \Cc \to \Dd$ fails to lift to an automorphism of $\Aa$. This isomorphism of commutative $C^\ast$-algebras sends $A_v$ to $\sigma_{e_v}^z$, where $e_v$ is the upper (vertical) edge of the star $v$, and $B_{\tilde v}$ to $\sigma_{e_{\tilde v}}^z$, where $e_{\tilde v}$ is the upper (horizontal) edge of the face $\tilde v$. 
\begin{proposition}\label{Prop:AutoM}
The isomorphism $\alpha\colon \Cc\to\Dd$ does not lift to an
automorphism of $\Aa$.
\end{proposition}

\begin{proof}
Suppose, by contradiction, that $\alpha$ lifts to an automorphism of $\Aa$. Fix a vertex $v$. Then there exists $X:=\alpha^{-1}(\sigma_{e_v}^x)\in\Aa,$ with $X\neq 0,$ such that
$$ A_vXA_v=-X,
    \qquad
    A_{v'}XA_{v'}=X
    \quad\text{for every }v'\neq v .$$
For $n\in\mathbb N^\times$, let $\Lambda_n$ be the set of edges contained in the box $\{-n,\ldots,n\}^{\times 2}+v$, and choose
$X_n\in\Aa_{\Lambda_n}$ with $X_n\to X$ in norm. Let
$\tilde\partial\Lambda_n$ be the dual path surrounding $\Lambda_n$, and set $
    \bar\Lambda_n
    :=
    \Lambda_n\cup
    \{e\in E:\ e\cap \tilde\partial\Lambda_n\neq\emptyset\}.$ Then
\[
    \prod_{A_u\in\Aa_{\bar\Lambda_n}} A_u
    =
    F^x_{\tilde\partial\Lambda_n}.
\]
Indeed, each interior edge appears in exactly two star operators and cancels, while each boundary edge appears only once, yielding $F^x_{\tilde\partial\Lambda_n}$. Since $F^x_{\tilde\partial\Lambda_n}$ is supported outside $\Lambda_n$, it
commutes with $X_n$. Hence
\[
    \Big(\prod_{A_u\in\Aa_{\bar\Lambda_n}} A_u\Big)
    X_n
    \Big(\prod_{A_u\in\Aa_{\bar\Lambda_n}} A_u\Big)
    =
    X_n .
\]
On the other hand, using the commutation relations satisfied by $X$, all
factors $A_u$ with $u\neq v$ commute with $X$, while $A_v$ anticommutes with
$X$. Therefore
\[
    \Big(\prod_{A_u\in\Aa_{\bar\Lambda_n}} A_u\Big)
    X
    \Big(\prod_{A_u\in\Aa_{\bar\Lambda_n}} A_u\Big)
    =
    A_vXA_v
    =
    -X .
\]
Consequently,
\[
    \|X+X_n\|
    =
    \left\|
    \Big(\prod_{A_u\in\Aa_{\bar\Lambda_n}} A_u\Big)
    (X-X_n)
    \Big(\prod_{A_u\in\Aa_{\bar\Lambda_n}} A_u\Big)
    \right\|
    =
    \|X-X_n\|.
\]
Since $X_n\to X$, we obtain $X_n\to -X$ as well. Hence $X=-X$, and therefore
$X=0$, a contradiction. Thus $\alpha$ cannot lift to an automorphism of
$\Aa$.
\end{proof}

% Suppose that $\alpha$ lifts to an automorphism of $\Aa$, and pick an arbitrary vertex $v$. Then there must exist $X = \alpha^{-1}(\sigma_{e_v}^x) \in \Aa$, $X \neq 0$, such that $A_v X A_v = -X$ and $X$ commutes with all other $A_{v'}$. For $n\in \NM^\times$, let $\{-n,\ldots,n\}^{\times 2}+v$ be a box of the square lattice surrounding $v$, and take $\Lambda_n$ to be the set of edges contained in such box. Then there should be a sequence $X_n \in \Aa_{\Lambda_n}$ converging to $X$ in norm. Now, take the dual path $\tilde \partial \Lambda_n$ surrounding $\Lambda_n$ and consider $\bar \Lambda_n = \Lambda_n \cup \{e \in E, \ e \cap \tilde \partial \Lambda_n \neq \emptyset\}$. Note that $\prod_{A_u \in \Aa_{\bar \Lambda_n}} A_u = F_{\tilde \partial \Lambda_n}^x$ and, as such, 
% \[
%     \big (\prod_{A_u \in \Aa_{\bar \Lambda_n}} A_u \big ) X_n \big (\prod_{A_u \in \Aa_{\bar \Lambda_n}} A_u \big ) = X_n.
% \]
% On the other hand, using the commutation relations for $X$, we have 
% \[
%     \big (\prod_{A_u \in \Aa_{\bar \Lambda_n}} A_u \big ) X \big (\prod_{A_u \in \Aa_{\bar \Lambda_n}} A_u \big ) = A_v X A_v = -X.
% \]
% Thus $X+X_n=0$,
% which automatically implies $X=0$ by taking $n \to \infty$. This is a contradiction and, as such, $\alpha$ cannot be lifted to an automorphism of $\Aa$.

\section*{Declarations}
The authors have no conflicting or competing interests to declare that are relevant to the content of this article. No data was produced or used for this work.

\end{document}